\documentclass[11pt,a4paper]{article}
\usepackage{amscd}
\usepackage{enumerate}

\usepackage{amsmath, amssymb, latexsym}
\usepackage{bbm}
\usepackage{authblk}
\usepackage{amsfonts}
\usepackage{amsthm}
\usepackage{relsize}
\usepackage{setspace}
\usepackage{geometry}
\usepackage{url}
\usepackage{xspace}
\usepackage{tocloft}
\usepackage{graphics}
\usepackage{graphicx}
\usepackage{lscape}
\usepackage{microtype}
\usepackage[T1]{fontenc}
\usepackage{ulem}
\usepackage{mathrsfs}

\usepackage[usenames, dvipsnames]{color}
\usepackage[utf8]{inputenc}
\usepackage{tikz}

\usepackage[pagebackref=true]{hyperref}
\usepackage[alphabetic]{amsrefs}
\usepackage[english]{babel}

\newtheorem{theorem}{Theorem}[section]
\newtheorem{proposition}[theorem]{Proposition}

\newtheorem{lemma}[theorem]{Lemma}

\theoremstyle{definition}
\newtheorem{definition}[theorem]{Definition}

\newtheorem{remark}[theorem]{Remark}

\newtheorem*{rema*}{Remark}

\theoremstyle{problem}

\newcommand{\Aut}{\mathrm{Aut}}
\newcommand{\Opp}{\mathrm{Opp}}
\newcommand{\Ker}{\mathrm{Ker}}
\newcommand{\proj}{\mathrm{proj}}
\newcommand{\RR}{\mathbf{R}}

\newcommand{\CAT}{\mathrm{CAT}}
\newcommand{\cat}{$\mathrm{CAT}(0)$\xspace}

\newcommand{\Ch}{\mathrm{Ch}}

\newcommand{\Stab}{\mathrm{Stab}}

\newcommand{\bd}{\partial}

\newcommand{\Id}{\operatorname{Id}}

\newcommand{\Top}{\operatorname{\mathrm{Top}}}

\newcommand{\A}{\mathrm{\bold{A}}}

\newcommand{\blue}[1]{{#1}}
\newcommand{\red}[1]{{#1}}
\newcommand{\green}[1]{{#1}}

\def\og{\leavevmode\raise.3ex\hbox{$\scriptscriptstyle\langle\!\langle$~}}
\def\fg{\leavevmode\raise.3ex\hbox{~$\!\scriptscriptstyle\,\rangle\!\rangle$}}

\def\ie{{\it i.e.,\/}\ }
\def\eg{{\it e.g.,\/}\ }

\def\lc{{\it l.c.\/}\ }

\def\truc{\unskip\kern 3pt\penalty 500
\hbox{\vrule\vbox to 5pt{\hrule width 4pt\vfill\hrule}\vrule}\kern
3pt}

\def\vect{\overrightarrow}

\def\parni{\par\noindent}

\def\Z{{\mathbb Z}}

\def\R{{\mathbb R}}

\def\A{{\mathbb A}}
\def\M{{\mathbb M}}

\newcommand{\g}[1]{\mathfrak{#1}} 

\def\qa{\alpha}     

\def\qd{\delta}

\def\qf{\varphi}

\def\qr{\rho}
\def\qs {\sigma}

\def\qx{\xi}

\def\QD{\Delta}
\def\QF{\Phi}

\def\QL{\Lambda}
\def\QO{\Omega}

\def\sha{{\mathcal A}}   

\def\shm{{\mathcal M}}

\def\sht{{\mathcal T}}

\title{Strongly transitive actions on affine ordered hovels}

\author[1]{Corina Ciobotaru\thanks{corina.ciobotaru@unifr.ch}}
\author[2,3]{Guy Rousseau\thanks{Guy.Rousseau@univ-lorraine.fr}}

\affil[1]{Universit\'e de Gen\`{e}ve, Section de math\'ematiques, 2--4 rue du Li\`{e}vre, 1211~Gen\`{e}ve~4, Switzerland.}
\affil[2]{Universit\'e de Lorraine, Institut \'Elie Cartan de Lorraine, UMR 7502, Vand\oe uvre-l\`es-Nancy, F-54506, France.}
\affil[3]{CNRS, Institut \'Elie Cartan de Lorraine, UMR 7502, Vand\oe uvre-l\`es-Nancy, F-54506, France.}

\date{June 21, 2016}

\begin{document}
\newcounter{qcounter}

\maketitle

\begin{abstract}
A hovel is a generalization of the Bruhat--Tits building that is associated with an almost split Kac--Moody group $G$ over a non-Archimedean local field. In particular, $G$ acts strongly transitively on its corresponding hovel $\Delta$ as well as on the building at infinity of $\Delta$, which is the twin building associated with $G$. In this paper we study strongly transitive actions of groups that act on affine ordered hovels $\Delta$ 
and give necessary and sufficient conditions such that the strong transitivity of the action on $\Delta$ is equivalent to the strong transitivity of the action of the group on its building at infinity $\partial \Delta$. Along the way a criterion for strong transitivity is given and the cone topology on the hovel is introduced. We also prove the existence of strongly regular hyperbolic automorphisms of the hovel, obtaining thus good dynamical properties on the building at infinity $\partial \Delta$.

\end{abstract}

\section{Introduction}

Hovels were introduced by Gaussent and Rousseau~\cite{GR08} and developed further by Rousseau~\cite{Rou11, Rou12}, Charignon~\cite{Cha}, Gaussent--Rousseau~\cite{GR14} and Bardy-Panse--Gaussent--Rousseau in~\cite{BPGR14}. By the Bruhat--Tits theory, to every semi-simple algebraic group over a non-Archimedean local field one associates a symmetric space, called the corresponding building. More generally, to an almost split Kac--Moody group $G$ over a field $K$ one associates a hovel, which gives a generalization of the Bruhat--Tits construction. As for buildings, a hovel is covered by apartments corresponding to the maximal split tori and every apartment is a finite dimensional real affine space endowed with a set of affine hyperplanes called walls. At a first glance, a hovel looks like the Bruhat--Tits building. 
Still, for the case of hovels, the set of all walls is not always a locally finite system of hyperplanes and  it is no longer true that any two points are contained in a same apartment. For this reason, the word ``building'' is replaced by the word ``hovel''.  Moreover, if we look at infinity of the hovel, we get a twin building, which is the analogue of the spherical building at infinity associated to the Bruhat--Tits building. 
As an application, Bardy-Panse, Gaussent and Rousseau~\cite{GR14},~\cite{BPGR14}, use the construction of the hovel to define the spherical Hecke and Iwahori--Hecke algebras associated to an almost split Kac--Moody group over a local non-Archimedean field. 
In \cite{GR08}, the hovel is used to make the link (in the representation theory of Kac--Moody groups) between the  Littelmann's path model and the Mirkovic--Vilonen cycle model.

As we mentioned above, to every almost split Kac--Moody group $G$ there is associated a natural symmetric space, the hovel, where the group $G$ acts by automorphisms and such that at infinity of the hovel we obtain the twin building corresponding to $G$.  Moreover, by construction, $G$ acts strongly transitively on its corresponding twin building and there is associated the spherical Hecke algebra, which is commutative (see Gaussent--Rousseau~\cite{GR14}). It is then natural to see to what extent the main result of Caprace--Ciobotaru~\cite{CaCi} can be obtained in the more general case of hovels. \red{For the convenience of the reader we recall this result and the definitions of strong transitivity that are involved.

\begin{definition}
Let $\Delta$ be a building \red{and $\sha$ be the (not necessarily complete) apartment system defining $\Delta$.} Let $G {\leq} \Aut(\Delta)$. We say that $G$ acts strongly transitively on $\Delta$ if $G$ acts transitively on the set  $\sha$ of apartments and if, for one apartment $A\in\sha$, the stabilizer $Stab_G(A)$ of $A$ in $G$ acts transitively on the chambers in $A$. Accordingly, the building at infinity $ \partial \Delta$ is defined with respect to $(\Delta, \sha)$.

When $\Delta= (\Delta_{-}, \Delta_{+})$ is a twin building and $G {\leq} \Aut(\Delta)$ we say that $G$ acts strongly transitively on $\Delta$ if $G$ acts transitively on the set of {all} twin apartments  $A=(A_-, A_+)$ of $\Delta$ and $\Stab_{G}(A)$ acts transitively on $\Ch(A_+)$ (so also on $\Ch(A_-)$), for every twin apartment  $A=(A_-, A_+)$ of $\Delta$.
\end{definition}}

\begin{theorem}(See \cite[Theorem~1.1]{CaCi})
\label{thm:main-thm_recall}
Let $G$ be a locally compact group acting continuously and  properly by type-preserving automorphisms on a locally finite thick Euclidean building $\Delta$ \red{(endowed with its complete apartment system)}.  Then the following are equivalent:
\begin{enumerate}[(i)]
\item
\label{thm:main-thm-i}
$G$ acts strongly transitively on $\Delta$;

\item
\label{thm:main-thm-ii}
$G$ acts strongly transitively on the spherical building at infinity $\partial \Delta$;

\item
\label{thm:main-thm-iii}
$G$ acts cocompactly on $\Delta$,  and there exists a compact subgroup $K$ of $G$ such that the Hecke algebra corresponding to $(G,K)$ is commutative (in this case the pair $(G,K)$ is called Gelfand).
\end{enumerate} 

In particular, when $G$ acts cocompactly on $\Delta$, it admits a Gelfand pair only when G acts strongly transitively on $\Delta$. In this case, its only Gelfand pairs are of the form $(G,G_x)$ where x is a special vertex of $\Delta$ and $G_x$ is the stabilizer of $x$ in $G$.
\end{theorem}

In order to generalize Theorem~\ref{thm:main-thm_recall} to hovels, one has to be aware of the fact that hovels are more difficult objects to deal with than affine buildings. The difficulties are coming from the fact that (local) chambers and (local) faces are no longer subsets of an apartment $\mathbb A$, but filters of subsets of $\mathbb A$, the hovel is not necessarily a geodesic metric space, two points are not always contained in the same apartment, and there is no natural topology on the hovel that can be induced on the full group of automorphisms. Therefore, under some assumptions, we obtain the following theorem, which is the main result of this article. In order to motivate these assumptions, we have:

\red{\begin{remark}
\label{rem::DTE_not_ST}
Let $T$ be a locally finite and regular tree and let $ \xi$ be an end of $T$. Define $\sha$ as the set of all apartments not containing $\xi$ and $G$ the group of automorphisms of $T$ fixing $\xi$. Then $G$ is $2$--transitive (equivalently, strongly transitive) on the
set $E(\sha)$ of ends of $(T,\sha)$ and for any $\eta \in E(\sha)$ the orbits of $G^{0}_{\eta}$ are closed in $E(\sha) \setminus \{\eta\} = E(T ) \setminus \{\eta, \xi\}$. But $G$ is not strongly transitive on $(T, \sha)$: any $g \in G$ stabilizing an apartment $A \in \sha$ fixes the projection of $\xi$ on $A$.
\end{remark}}

\red{\begin{definition}
\label{def::LST}
Let $(\Delta, \sha)$ be an affine ordered hovel \green{that is not a tree} and let $G$ be a vectorially Weyl subgroup of $\Aut(\Delta)$. Let $\sigma, \sigma' \subset \partial \Delta$ be a pair of opposite panels and denote by $P(\sigma, \sigma')$ the union of all apartments of $(\Delta, \sha)$ whose boundaries contain $\sigma$ and $\sigma'$. By Lemma~\ref{lem:TreeWalls} consider the panel tree $\Delta(\sigma, \sigma'):= (T(\sigma, \sigma'), \sha(\sigma, \sigma'))$ associated with $\sigma$ and $\sigma'$.  Let $G_{\sigma, \sigma'}$ be the stabilizer in $G$ of $\sigma, \sigma'$. We say that $G$ satisfies \textbf{condition (LST)} if for each pair of opposite panels $(\sigma, \sigma')$ the subgroup $G_{\sigma, \sigma'}$ acts strongly transitively on the panel tree $\Delta(\sigma, \sigma')$.
\end{definition}}

\par\blue{If we assume the apartment system "locally complete" (see section \ref{1.7} below) and $G$ strongly transitive on $\partial \Delta$, then Lemma \ref{lem:loc-comp} tells that $G$ satisfies (LST). The Remark \ref{rem::DTE_not_ST} proves that this conclusion may be wrong without an hypothesis on $\sha$. }

\begin{theorem}
\label{thm::main_theorem_intr}
Let $\red{(\Delta, \sha)}$ be a semi-discrete thick, affine ordered hovel, \red{that is not a tree.} 
 Let $G$ be a vectorially Weyl subgroup of $\Aut(\Delta)$ \red{satisfying condition (LST)} and with the property that for every $c \in \Ch(\partial \Delta)$ the $G^{0}_c$--orbits on $\Opp(c)$ are closed with respect to the cone topology on $\Ch(\partial \Delta)$. 
If $\QD$ has some factors of affine type, we ask moreover that $G$ satisfies the condition (AGT) of Remark~\ref{rem::AGT}. 
Then the following are equivalent:

\begin{enumerate}[(i)]
\item
\label{thm:main-thm-i}
$G$ acts strongly transitively on $\Delta$;

\item
\label{thm:main-thm-ii}
$G$ acts strongly transitively on the twin building at infinity $\partial \Delta$.
\end{enumerate}
\end{theorem}

\par The definition of a strongly transitive action of a subgroup of $\Aut(\Delta)$ on a hovel $\Delta$ is given in Rousseau~\cite[4.10]{Rou13} (see also Gaussent--Rousseau~\cite[1.5]{GR14}, or  Definition~\ref{4.2} below). Because, in principle, this definition is difficult to be verified, Proposition~\ref{lem:ST:criterion} from Section~\ref{s4} provides us with two equivalent simpler definitions, which will prove to be very useful in the sequel.

\medskip
To prove Theorem~\ref{thm::main_theorem_intr} (see Section~\ref{sec::main_theorem}), we have followed the main ideas as in Caprace--Ciobotaru~\cite{CaCi} and their results are carefully translated in the language of affine ordered hovels. Still, one has to overcome the differences between affine ordered hovels and affine buildings. Firstly, let us recall the main strategy in order to prove Theorem~\ref{thm::main_theorem_intr}. The implication (i) $\Rightarrow$ (ii) is easy. The converse implication (ii) $\Rightarrow$ (i) is more involved. 
Using Proposition~\ref{lem:ST:criterion}, we only have to prove the following: the pointwise stabilizer $G^{0}_{c} \leq G$ of any sector germ (i.e., chamber at infinity) $c$ of $\partial \Delta$, is transitive on the set of all apartments containing $c$. To do so,  more ingredients are needed: the cone topology and the existence and the dynamics of strongly regular elements.

Section~\ref{sec::cone_top_hovel} defines the cone topology only on the set of chambers at infinity of an affine ordered hovel $\Delta$ and verifies that this topology does not depend on the chosen base point, the main difficulty being that a hovel is not necessarily a geodesic metric space as in the case of \cat spaces. 

The existence and the dynamics of strongly regular elements are treated in Section~\ref{sec::str_reg_elements}. Although the proof of the dynamics of strongly regular elements at infinity of the hovel is the same as in the case of affine buildings from Caprace--Ciobotaru~\cite[Section~2.1]{CaCi}, the existence of those elements requires a different proof than in the latter case.

Finally, the dynamics of strongly regular elements and the assumption that for every $c \in \Ch(\partial \Delta)$, the $G^{0}_c$--orbits on $\Opp(c)$ are closed with respect to the cone topology on $\Ch(\partial \Delta)$, allows us to conclude that indeed $G^0_{c} $ is transitive on the set of all apartments containing $c$ and Theorem~\ref{thm::main_theorem_intr} follows.
The mysterious condition about the closedness of $G^{0}_c$--orbits on $\Opp(c)$ is imposed in order to substitute the lack of the topology induced from the hovel $\Delta$ to the group $\Aut(\Delta)$. Notice that when $\Delta$ is a locally finite affine building, the group $\Aut(\Delta)$ is locally compact and the group $G$ appearing in Theorem~\ref{thm:main-thm_recall} is a closed subgroup of $\Aut(\Delta)$. This implies that $G^{0}_c$ is a closed subgroup of $G$ and also that $G^{0}_c$--orbits on $\Opp(c)$ are closed with respect to the cone topology.

\subsection*{Acknowledgement} Part to this work was conducted when the second named author was visiting the University of Geneva. We would like to thank this institution for its hospitality and good conditions of working. \green{We also thank Bernhard  M\"{u}hlherr for his remarks on this paper.}

\section{Affine ordered hovels}\label{s1}


\subsection{Vectorial data}\label{1.1}  We consider a quadruple $(V,W^v,(\qa_i)_{i\in I}, (\qa^\vee_i)_{i\in I})$ where $V$ is a finite dimensional real vector space, $W^v$ a subgroup of $GL(V)$ (the vectorial Weyl group), $I$ a finite non empty set, $(\qa^\vee_i)_{i\in I}$ a family in $V$ and $(\qa_i)_{i\in I}$ a free family in the dual $V^*$.
 We ask these data to satisfy the conditions of Rousseau~\cite[1.1]{Rou11}.
  In particular, the formula $r_i(v)=v-\qa_i(v)\qa_i^\vee$ defines a linear involution in $V$ which is an element in $W^v$ and $(W^v,\{r_i\mid i\in I\})$ is a Coxeter system.

  \par To be more concrete, we consider the Kac--Moody case of [\lc; 1.2]: the matrix $\M=(\qa_j(\qa_i^\vee))_{i,j\in I}$ is a generalized (eventually decomposable) Cartan matrix.
  Then $W^v$ is the Weyl group of the corresponding Kac--Moody Lie algebra $\g g_\M$ and the associated real root system is
$$
\QF=\{w(\qa_i)\mid w\in W^v,i\in I\}\subset Q=\bigoplus_{i\in I}\,\Z.\qa_i.
$$ We set $\QF^\pm{}=\QF\cap Q^\pm{}$ where $Q^\pm{}=\pm{}(\bigoplus_{i\in I}\,(\Z_{\geq 0}).\qa_i)$ and $Q^\vee=(\bigoplus_{i\in I}\,\Z.\qa_i^\vee)$, $Q^\vee_\pm{}=\pm{}(\bigoplus_{i\in I}\,(\Z_{\geq 0}).\qa_i^\vee)$.
   We have  $\QF=\QF^+\cup\QF^-$ and, for $\qa=w(\qa_i)\in\QF$, $r_\qa=w.r_i.w^{-1}$ and $r_\qa(v)=v-\qa(v)\qa^\vee$, where the coroot $\qa^\vee=w(\qa_i^\vee)$ depends only on $\qa$.  We shall also consider the imaginary roots of $\g g_\M$: $\QF_{im}=\QF_{im}^+\cup\QF_{im}^-$ where $\QF_{im}^\pm\subset Q^\pm$ 
   is $W^v-$stable and $\QF\sqcup\QF_{im}$ is the set of all roots of  $\g g_\M$, associated to some Cartan subalgebra.


  \par The {\bf fundamental positive chamber} is $C^v_f=\{v\in V\mid\qa_i(v)>0,\forall i\in I\}$.
   Its closure $\overline{C^v_f}$ is the disjoint union of the vectorial faces $F^v(J)=\{v\in V\mid\qa_i(v)=0,\forall i\in J,\qa_i(v)>0,\forall i\in I\setminus J\}$ for $J\subset I$.
    The positive (resp. negative) vectorial faces are the sets $w.F^v(J)$ (resp. $-w.F^v(J)$) for $w\in W^v$ and $J\subset I$; they are chambers (resp., panels) when $J=\emptyset$ (resp. $\vert J\vert=1$).
    The support of such a face is the vector space it generates.
    The set $J$ or the face $w.F^v(J)$ or an element of this face is called {\bf spherical} if the group $W^v(J)$ generated by $\{r_i\mid i\in J\}$ is finite.  
    A chamber or a panel is spherical.

 \par The {\bf Tits cone}  $\sht$ (resp., its {\bf interior}  $\sht^\circ$) is the (disjoint) union of the positive vectorial (resp., and spherical) faces. Both are $W^v-$stable convex cones in $V$.
    
\par We make no irreducibility hypothesis for $(V,W^v)$. So $V$ (and also $\A$, $\QD$ below) may be a product of direct irreducible factors, which are either of finite, affine or indefinite type, see Kac~\cite[4.3]{K90}.

\subsection{The model apartment}\label{1.2} As in Rousseau~\cite[1.4]{Rou11} the model apartment $\A$ is $V$ considered as an affine space and endowed with a family $\shm$ of walls. 
 These walls  are affine hyperplanes directed by Ker$(\qa)$, for $\qa\in\QF$.
 They can be described as $M(\qa,k)=\{v\in V\mid\qa(v)+k=0\}$, for $\qa\in\QF$ and $k\in\QL_\qa\subset\R$ (with $\vert\QL_\qa\vert=\infty$). 
 We may (and shall) suppose the origin to be \textbf{special}, \ie $0\in\QL_\qa$, $\forall\qa\in\QF$.

 \par We say that this apartment is {\bf semi-discrete} if each $\QL_\qa$ is discrete in $\R$; then 
 $\QL_\qa=k_\qa.\Z$ is a non trivial discrete subgroup of $\R$.
Using the  Lemma 1.3 in Gaussent--Rousseau~\cite{GR14} (\ie replacing $\QF$ by another system $\QF_1$) we may assume that $\QL_\qa=\Z, \forall\qa\in\QF$.

  \par For $\qa\in\QF$, $k\in\QL_\qa$ and $M=M(\qa,k)$, the reflection $r_{\qa,k}=r_M$ with respect to the wall $M$ is the affine involution of $\A$ with fixed points this wall $M$ and associated linear involution $r_\qa$.
   The affine Weyl group $W^a=W^v\ltimes Q^\vee$ is the group generated by the reflections $r_M$ for $M\in \shm$; we assume that $W^a$ stabilizes $\shm$.

   \par An automorphism of $\A$ is an affine bijection $\qf:\A\to\A$ stabilizing the set of pairs $(M,\qa^\vee)$ of a wall $M$ and the coroot associated with $\qa\in\QF$ such that $M=M(\qa,k)$, $k\in\Z$. The group $\Aut(\A)$ of these automorphisms contains $W^a$ and normalizes it.

   \par For $\qa\in\QF$ and $k\in\R$, $D(\qa,k)=\{v\in V\mid\qa(v)+k\geq 0\}$ is an half-space, it is called an {\bf half-apartment} if $M(\qa,k)$ is a wall \ie $k\in\QL_\qa$. 


The Tits cone $\mathcal T$ 
and its interior $\mathcal T^o$
are convex and $W^v-$stable cones, therefore, we can define two $W^v-$invariant preorder relations  on $\mathbb A$: 
$$
x\leq y\;\Leftrightarrow\; y-x\in\mathcal T
; \quad x\stackrel{o}{<} y\;\Leftrightarrow\; y-x\in\mathcal T^o.
$$
 If $W^v$ has no  fixed point in $V\setminus\{0\}$ and no finite factor, then they are order relations; still, they are not in general: one may have $x\leq y,y\leq x$ and $x\neq y$. 


\subsection{Faces, sectors, chimneys...}
\label{suse:Faces}

 The faces in $\mathbb A$ are associated to the above systems of walls
and half-apartments
. As in Bruhat--Tits~\cite{BrT72}, they
are no longer subsets of $\mathbb A$, but filters of subsets of $\mathbb A$. For the definition of that notion and its properties, we refer to Bruhat--Tits~\cite{BrT72} or Gaussent--Rousseau~\cite[Definition~2.3]{GR08}. We endow $\A$ with its affine space topology.

If $F$ is a subset of $\mathbb A$ containing an element $x$ in its closure
, \textbf{the germ} of $F$ in $x$ is the filter $\mathrm{germ}_x(F)$ consisting of all subsets of $\mathbb A$ containing the intersection of $F$ and some neighborhoods of $x$. 

Given a filter $F$ of subsets of $\mathbb A$, its {\bf enclosure} $cl_{\mathbb A}(F)$ (resp., {\bf closure} $\overline F$) is the filter made of  the subsets of $\mathbb A$ containing an element of $F$ of the shape $\cap_{\alpha\in\QF\cup\QF_{im}}\,D(\alpha,k_\alpha)$, where $k_\alpha\in\QL_\qa\cup\{\infty\}$ (resp., containing the closure $\overline S$ of some $S\in F$). 
Its \textbf{support}, denoted by $supp( F )$, is the smallest affine subspace of $\A$ containing it.

\medskip

A {\bf local face} $F$ in the apartment $\mathbb A$ is associated
 to a point $x\in \mathbb A$ (its vertex) and a  vectorial face $F^v$ in $V$ (its direction); it is $F=F^\ell(x,F^v):=germ_x(x+F^v)$.
 Its closure is $\overline{F^\ell}(x,F^v)=germ_x(x+\overline{F^v})$. 
 The enclosure  $cl_{\mathbb A}(F)$ is called a closed face.

There is an order on the local faces: the assertions ``$F$ is a face of $F'$ '',
``$F'$ covers $F$ '' and ``$F\leq F'$ '' are by definition equivalent to
$F\subset\overline{F'}$.
 The dimension of a local face $F$ is the dimension of 
  its support; if $F=F^\ell(x,F^v)$, then we define $supp(F):=x+supp(F^v)$.


 A {\bf local chamber}  is a maximal local face, \ie a local face $F^\ell(x,\pm w.C^v_f)$ for $x\in\A$ and $w\in W^v$.
 The {\bf fundamental local chamber} is $C_0^+=germ_0(C^v_f)$.  A {\bf local panel}  is a local face $F^\ell(x,F^v)$, where $x\in\A$ and $F^v$ is a vectorial panel.


\medskip
 A {\bf sector} in $\mathbb A$ is a $V-$translate $Q=x+C^v$ of a vectorial chamber
$C^v=\pm w.C^v_f$ ($w \in W^v$), $x$ is its {\bf base point} and $C^v$ is its  {\bf direction}; it is open in $\A$.  
Two sectors have the same direction if, and only if, they are conjugate
by a $V-$translation, and if, and only if, their intersection contains another sector.

 The {\bf sector-germ} of a sector $Q=x+C^v$ in $\mathbb A$ is the filter $germ(Q)$ of
subsets of~$\mathbb A$ consisting of the sets containing a $V-$translation of $Q$, it is well
determined by the direction $C^v$.
So, the set of
translation classes of sectors in $\mathbb A$, the set of vectorial chambers in $V$ and
 the set of sector-germs in $\mathbb A$ are in canonical bijection.

 A {\bf sector-face} in $\mathbb A$ is a $V-$translation $\mathfrak f=x+F^v$ of a vectorial face
$F^v=\pm w.F^v(J)$. The sector-face-germ of $\mathfrak f$ is the filter $\mathfrak F$ of
subsets containing a shortening of $\mathfrak f$ \ie a translation $\mathfrak f'$ of $\mathfrak f$ by a vector in $F^v$ ({\it i.e.,} $\mathfrak
f'\subset \mathfrak f$). If $F^v$ is spherical, then $\mathfrak f$ and $\mathfrak F$ are also called
spherical. The sign of $\mathfrak f$ and $\mathfrak F$ is the sign of $F^v$.
We say that $\mathfrak f$ (resp. $\mathfrak F$) is a {\bf sector-panel} (resp., {\bf sector-panel-germ}) if, and only if, $F^v$ is a vectorial panel.

\medskip
A {\bf chimney} in $\mathbb A$ is associated to a local face $F=F^\ell(x, F_0^v)$, called its basis, and to a vectorial face $F^v$, its direction; it is the filter
$$
\mathfrak r(F,F^v) = cl_{\mathbb A}(F+F^v).
$$ A chimney $\mathfrak r = \mathfrak r(F,F^v)$ is {\bf splayed} if $F^v$ is spherical, it is {\bf solid} if its support (as a filter) has a finite  pointwise stabilizer in $W^v$. 
A splayed chimney is therefore solid. The enclosure of a sector-face $\mathfrak f=x+F^v$ is a chimney.
The germ of the chimney $\g r$ is the filter $\mathrm{germ}(\mathfrak r )=\mathfrak R$ consisting of all subsets of $\mathbb A$ which contain $ \mathfrak r(F+\qx,F^v)$  for some $\qx\in F^v$.

 \par A  ray $\delta$ with origin in $x$ and containing $y\not=x$ (or the interval $(x,y]$, or the segment $[x,y]$) is called {\bf preordered} if $x\leq y$ or $y\leq x$ and {\bf generic} if $x\stackrel{o}{<} y$ or $y\stackrel{o}{<} x$; its sign is $+$ if $x\leq y$ and $-$ if $x\geq y$.
 The germ of this ray is the filter $\mathrm{germ}(\qd)$ consisting of all subsets of $\mathbb A$ which contain a shortening of $\qd$ \ie $\qd\setminus[x,z)$  for some $z\in\qd$.

 \subsection{The hovel}\label{1.3}

 In this section, we recall the definition and some properties of an affine ordered hovel given  in Rousseau~\cite{Rou11}.

\medskip
\parni{\bf 1)} An apartment of type $\mathbb A$ is a set $A$ endowed with a set $Isom^w\!(\mathbb A,A)$ of bijections (called \textbf{Weyl-isomorphisms}) such that, if $f_0\in Isom^w\!(\mathbb A,A)$, then $f\in Isom^w\!(\mathbb A,A)$ if, and only if, there exists $w\in W^a$ satisfying $f = f_0\circ w$.
An \textbf{isomorphism} (resp., a \textbf{Weyl-isomorphism}, resp., a \textbf{vectorially Weyl-isomorphism}) between two apartments $\varphi :A\to A'$ is a bijection such that, for any $f\in Isom^w\!(\mathbb A,A)$, $f'\in Isom^w\!(\mathbb A,A')$, $f'^{-1}\circ\qf\circ f\in \Aut(\A)$ (resp., $\in W^a$, resp., $\in(W^v\ltimes V)\cap \Aut(\A)$).
We write $Isom(A,A')$ (resp., $Isom^w(A,A')$, resp., $Isom^w_\R(A,A'))$ for the set of these isomorphisms.
 As the filters in $\A$ defined in Section~\ref{suse:Faces} above (\eg local faces, sectors, walls,...) are permuted by $\Aut(\A)$, they are well defined in any apartment of type $\A$ and exchanged by any isomorphism.

\begin{definition}
\label{de:AffineHovel}
An \textbf{affine ordered hovel of type $\mathbb{A}$} is a set $\QD$ endowed with a covering $\mathcal{A}$ of subsets  called apartments such that:
\begin{enumerate}
\item[{\bf (MA1)}] any $A\in \mathcal A$ admits a structure of an apartment of type $\mathbb A$;

\item[{\bf (MA2)}] if $F$ is a point, a germ of a preordered interval, a generic  ray or a solid chimney in an apartment $A$ and if $A'$ is another apartment containing $F$, then $A\cap A'$ contains the
enclosure $cl_A(F)$ of $F$ and there exists a Weyl-isomorphism from $A$ onto $A'$ fixing (pointwise) $cl_A(F)$;

\item[{\bf (MA3)}] if $\mathfrak R$ is  the germ of a splayed chimney and if $F$ is a closed face or a germ of a solid chimney, then there exists an apartment that contains $\mathfrak R$ and $F$;

\item[{\bf (MA4)}] if two apartments $A,A'$ contain $\mathfrak R$ and $F$ as in {\bf (MA3)}, then their intersection contains $cl_A(\mathfrak R\cup F)$ and there exists a Weyl-isomorphism from $A$ onto $A'$ fixing (pointwise) $cl_A(\mathfrak R\cup F)$;

\item[{\bf (MAO)}] if $x,y$ are two points contained in two apartments $A$ and $A'$, and if $x\leq_A y$ then the two line segments $[x,y]_A$ and $[x,y]_{A'}$ are equal.
\end{enumerate}
\end{definition}

\par\blue{Remark that we assume no completeness property for the apartment system $\sha$.}

  \par We say that $\QD$ is \textbf{thick} (resp., \textbf{thick of finite thickness}) if the number of local chambers  containing a given (local) panel is $\geq 3$ (resp., and finite).

  \par An \textbf{automorphism} (resp., a \textbf{Weyl-automorphism}, resp., \textbf{a vectorially Weyl-automorphism}) of $\QD$ is a bijection $\qf:\QD\to\QD$ such that $A\in\sha\iff \qf(A)\in\sha$ and then $\qf\vert_A:A\to\qf(A)$ is an isomorphism (resp., a Weyl-isomorphism, resp., a vectorially Weyl-isomorphism).

\medskip
\parni{\bf 2) The building at infinity}: By (MA3), two spherical sector-faces (or generic rays) are, up to shortening, contained in a same apartment $A$; we say they are \textbf{parallel} if one of it is a translation of the other one. This does not depend on the choice of $A$ by (MA4). The parallelism is an equivalence relation. The parallel class $\partial\qd$ of a generic ray $\qd$ is called an \textbf{ideal point or a point at infinity}. The parallel class $\partial\mathfrak f$ of a spherical sector-face $\mathfrak f$ is called \textbf{an ideal face or a face at infinity} (a \textbf{chamber} if $\mathfrak f$ is a sector and a \textbf{panel} if $\mathfrak f$ is a sector-panel). Actually, a chamber at infinity is nothing else than a sector-germ.

Notice that, by the hypotheses made with respect to the definition of an affine ordered hovel $\Delta$, the rank of its building at infinity $\partial \Delta$ is strictly positive.

\par We write $\partial\g f\leq \partial\g f'$ if, for good choices of $\g f,\g f'$ in their parallel classes, we have $\g f\subset\overline{\g f'}$.
The ordered set of ideal faces of sign $\pm$ is (the set of spherical faces of) a building $\partial^\pm\QD$ with Weyl group $W^v$; these buildings are twinned (see Rousseau~\cite[3.7]{Rou11}).  We write $\partial\QD=\partial^+\QD\sqcup\partial^-\QD$ and $Ch(\partial\QD)=\{\mathrm{ideal\ chambers}\}$.
 
 \par We say that an ideal point $\qx$, an ideal face $\phi$, is at infinity of an apartment $A$ (or a wall $h$, an half apartment $H$, ...) if we may write $\qx=\partial\qd$, $\phi=\partial\g f$, with $\qd\subset A$, $\g f\subset A$ (or $\subset h$, $\subset H$, ...).
 We say that $\qx=\partial\qd\in\phi=\partial\g f$ if $\qd\subset\g f$.
 
 \par From (MA3) and (MA4), we see that a point $x\in\QD$ and an ideal chamber $c$ (resp., an ideal point $\qx$) determine, in the parallel class $c$ (resp., $\qx$), a unique sector (resp., generic ray $[x,\qx)$) of base point (resp., origin) $x$. To fix the notation, for a point $x \in \QD$ and an ideal chamber $c \in \partial \Delta$, we denote by $Q_{x,c}$ the sector in $\Delta$ with base point $x$ that corresponds to the chamber at infinity $c$.
 
\par Any vectorially Weyl-automorphism $\qf$ of $\QD$ acts on $\partial\QD$ as a type preserving automorphism; if it stabilizes an ideal face $\phi$, it fixes any ideal point $\qx\in\phi$.

\medskip
\parni{\bf 3)} For $x,y\in\QD$, we introduce the relation:

\par $x\leq y$ if, and only if, there is an (or for any) apartment $A$ such that $x,y\in A$ and $x\leq_Ay$ (\ie $f^{-1}(x)\leq f^{-1}(y)$ for any $f\in Isom^w(\A,A)$).

\par This relation is a preorder on $\QD$, invariant by any vectorially Weyl-automorphism.

\medskip
\parni{\bf 4)} Let $c$ be an ideal chamber at infinity of an apartment $A$. By (MA3) and (MA4), for any $x\in\QD$, there is an apartment $A'$ containing $x$ and the sector-germ $germ(Q)$ associated to $c$ and there is a unique Weyl-isomorphism $\qf:A\to A'$ fixing $germ(Q)$.
 So, by the usual arguments, we see that $x\mapsto\qf(x)$ is a well defined map $\qr_{A,c}:\QD\to A$, called the {\bf retraction of $\QD$ onto $A$ with center $c$}.

\medskip
\parni{\bf 5)} The main examples of thick, affine, ordered, semi-discrete (resp., and of finite thickness) hovels are provided by the hovels of almost split Kac--Moody groups over fields complete for a discrete valuation and with a perfect (resp., finite) residue field, see \cite{GR08}, \cite{Rou12}, \cite{Cha} and \cite{Rou13}.

\subsection{Trees in affine ordered hovels}\label{1.5}

The following property of affine ordered hovels is useful in the sequel. 

\begin{lemma}
\label{lem:TreeWalls}
Let $\Delta$ be an  affine ordered hovel of dimension $n$. 
Let $\sigma, \sigma' \subset \partial \Delta$ be a pair of opposite panels at infinity. We denote by $P(\sigma, \sigma')$ the union of all apartments of $\Delta$ whose boundaries contain $\sigma$ and $\sigma'$. Then  $P(\sigma, \sigma')$ is a closed convex subset of $\Delta$, which splits canonically as a product
$$ P(\sigma, \sigma') \cong T \times \R^{n-1},$$
where $T$ is a  $\R-$tree whose ends are canonically in one-to-one correspondence with the elements of the set $\Ch(\sigma)$ of all ideal chambers having $\sigma$ as a panel. Under this isomorphism, the walls of $\Delta$, contained in $ P(\sigma, \sigma')$ and containing $\qs,\qs'$ at infinity, correspond to the subsets of the form $\{v\} \times \R^{n-1}$ with $v$ a vertex of $T$.

\par When $\QD$ is semi-discrete (resp., thick, resp., thick of finite thickness), then the $\R-$tree $T$ is a genuine discrete (resp., thick, resp., thick of finite thickness) tree.
\end{lemma}

\begin{proof}
A reference for this is Rousseau~\cite[Section 4.6]{Rou11}.
\end{proof}

\subsection{Stabilizers of pairs of opposite panels}\label{1.6}

\par Recall that a group of automorphisms of a twin building is said to be strongly transitive if, and only if, it acts transitively on the pairs $(c,A)$, where $A$ is an apartment of the twin building and $c$ is a chamber of $A$.

\medskip
The following result is well-known; still, we provide a proof for completeness. 

\begin{lemma}
\label{lem:OppPanels}
Let $Z$ be a thick twin building and let $G \leq \Aut(Z)$ be a strongly transitive group of type-preserving automorphisms. Then, for any pair of opposite panels $\sigma, \sigma'$, the stabilizer $G_{\sigma, \sigma'}$ is $2$-transitive on the set of chambers $\Ch(\sigma)$. 
\end{lemma}

\begin{proof}
Let $c \in \Ch(\sigma)$ and $x, y \in \Ch(\sigma)$ be two chambers different from $c$. Let $c' = \proj_{\sigma'}(c)$. Then $x$ and $y$ are both opposite $c'$. Therefore there is $g \in G_{c'}$ mapping $x$ to $y$. Since $G$ is type-preserving, it follows that $g$ fixes $\sigma'$ (because it fixes $c'$) and hence $\sigma$ (because it is the unique panel of $x$, respectively $y$, which is opposite $\sigma'$). Thus $g \in G_{\sigma, \sigma'}$. Moreover $g$ fixes $c'$, and hence also $c = \proj_\sigma(c')$. 

Thus, for any triple $c, x, y$ of distinct chambers in $\Ch(\sigma)$, we have found an element $g \in G_{\sigma, \sigma'}$ fixing $c$ and mapping $x$ to $y$. The $2$-transitivity of $G_{\sigma, \sigma'}$ on $\Ch(\sigma)$ follows. 
\end{proof}

\blue{\subsection{Locally \green{complete} apartment systems}
\label{1.7}

\par There is no existing definition of complete apartment systems for hovels. We introduce now an, a priori, weaker definition:

\begin{definition}\label{8.1} {Let $\QD$ be an affine ordered hovel.} The apartment system $\sha$ of $\QD$ is said to be \textbf{locally-complete} if the following holds:

\begin{enumerate} Take any increasing sequence $\{H_n\}_{n\geq 0}$ of half-apartments that are respectively contained in apartments $\{A_n\}_{n \geq 0}\in\sha$.
We suppose that, for an ideal chamber $c\in \Ch(\partial H_0)$, we have $\bigcup_{n\geq0}\,\qr_{A_0,c}(H_n)=A_0$.
Then $\bigcup_{n\geq0}\,H_n$ is an apartment $A\in\sha$.
\end{enumerate}

\end{definition}
\label{def::local_comp}

\par Clearly, for trees, the local-completeness is equivalent to the completeness of the apartment system. So, combining Lemmas \ref{lem:TreeWalls}, \ref{lem:OppPanels} and Corollary 3.6 in \cite{CaCi} (which assumes the apartment system complete), we get the following:

\begin{lemma}\label{lem:loc-comp} Let $(\QD,\sha)$ be an affine ordered hovel with a locally complete apartment system. Then any vectorially Weyl subgroup of $Aut(\Delta)$ that acts strongly transitively on $\partial\Delta$ satisfies condition (LST).
\end{lemma}
}

\section{Stabilizers of chambers at infinity}\label{s2}

This section reproduces the same results as in Caprace--Ciobotaru~\cite[Section 3]{CaCi}, where affine Euclidean buildings are studied. We only translate those results in the language of affine ordered hovels.

\begin{lemma}
\label{lem:unip}
Let  $\Delta$ be an affine ordered hovel and let $G \leq \Aut(\Delta)$ be any group of vectorially Weyl-automorphisms. Let $c \in \Ch(\partial \Delta)$ be a chamber at infinity. Then the set 
$$
G_c^0 :=\{g \in G_c \; | \;  g \text{ fixes some point of }\Delta\}
$$
is a normal subgroup of the stabilizer $G_c :=\{g \in G \; | \;  g(c)=c\}$. 
\end{lemma}

\begin{proof}
It is clear that $G_c, G_c^0$ are subgroups of $G$, as an element $g \in G_c$ fixing a point $x$ of $\Delta$ will also fix (pointwise) an entire sector that emanates from $x$ and pointing to $c$.  This is also used to prove that $G_c^{0}$ is normal in $G_c$ (see ~\cite[Section 3]{CaCi}).
\end{proof}

\begin{lemma}
\label{lem:BusemanRetraction}
Let  $\Delta$ be an affine ordered hovel and let $G \leq \Aut(\Delta)$ be any group of vectorially Weyl-automorphisms. Let $c \in \Ch(\partial \Delta)$ be a chamber at infinity and $A$ be an apartment of $\Delta$, whose boundary contains $c$. Then for any $g \in G_c$, the map 
$$\beta_c(g) \colon A \to A : x \mapsto \rho_{A, c}(g(x))$$
is an automorphism of the apartment $A$, acting as a (possibly trivial) translation. Moreover the map
$$\beta_c \colon G_c \to \Aut(A)$$
is a group homomorphism whose kernel coincides with $G_c^0$.
\end{lemma}

\begin{proof}
For $g\in G_c$, $g(A)$ is an apartment of $\Delta$ that contains the chamber $c$ in its boundary at infinity $\partial g(A)$.
 By the definition of the retraction $\rho_{A, c}$, as $g$ is vectorially Weyl and $c \in \Ch(\partial A) \cap \Ch(\partial g(A))$, it is easy to see that indeed $\beta_c(g) $ is an element of $\Aut(A)$, that can act as a translation or  can fix a point of $A$.

Let us prove that $\beta_c$ is a group homomorphism, whose kernel coincides with $G_c^0$. Let $g, h \in G_c$. Because $c \in \Ch(\partial A) \cap \Ch(\partial g(A)) \cap Ch(\partial h(A))$, there exists a common sector $Q_{x,c}$ in $\Ch(\partial A) \cap \Ch(\partial g(A)) \cap Ch(\partial h(A))$. This implies that for $\beta_c$ to be a group homomorphism it is enough to consider a point $y$ far away in the interior of the sector  $Q_{x,c}$ such that $ \rho_{A,c}(h(y))= h(y)$. We have that $$\beta_c(g) \beta_c(h) (y)= \beta_c(g) (\rho_{A, c}(h(y)))= \beta_c(g)(h(y))=\rho_{A, c}(g(h(y)))=\beta_c(gh)(y).$$

It is clear from the definition that $G_c^0$ is contained in the kernel of $\beta_c$. Let now $g \in \Ker(\beta_c)$. This implies that $g$ fixes a point in the intersection $A \cap g(A)$; therefore, $g \in G_{c}^{0}$. The conclusion follows.
\end{proof}

\begin{lemma}
\label{lem:Levi}
Let $\Delta$ be an affine ordered hovel and let $G \leq \Aut(\Delta)$ be a subgroup of vectorially Weyl-automorphisms that acts strongly transitively on $\partial \Delta$.

Then for any pair $c, c'$ of opposite chambers at infinity, every $G_c^0$--orbit on the set of chambers opposite $c$ is invariant under $G_{c, c'}:=G_c\cap G_{c'}$. 
\end{lemma}

\begin{proof}
Let $c, c'$ be a pair of opposite chambers at infinity of $\Delta$ and let $d  \in \Opp(c)$. As $G$ acts strongly transitively on $\partial \Delta$ there is some $g \in G_c$ with $d = g(c')$. As the quotient $G_c/G^0_c$ is abelian by Lemma~\ref{lem:BusemanRetraction}, we have that the subgroup $H:=G_{c, c'}G_{c}^{0} $ is normal in $G_c$ and 
$$\begin{array}{rcl}
G_{c, c'}(G^0_c(d)) & = & H(g(c')) \\
&= & g(H(c'))\\
& = & g(G^0_c.G_{c, c'}(c'))\\
& = & gG^0_c(c')\\
& = & G^0_c(g(c'))\\
&=& G^0_c(d).
\end{array}
$$
Thus the $ G^0_c$--orbit of $d$ is indeed invariant by $G_{c, c'}$, as claimed. It follows that the $H$--orbits on $\Opp(c)$ coincide with the $G^0_c$--orbits.
\end{proof}

\section{The cone topology on affine ordered hovels}
\label{sec::cone_top_hovel}

Let $\Delta$ be an affine ordered hovel. As in the case of $\CAT(0)$ spaces we would like to define a topology on the boundary $\partial \Delta $ of the hovel $\Delta$, which does not depend on the chosen base point. Recall that a hovel is not necessarily a geodesic metric space, therefore, we cannot apply the \cat theory. For the purpose of this article, it would be enough to define a cone topology only on the set of chambers at infinity; this set is denoted by $\Ch(\bd \Delta)$. 

\par In addition, by a chamber at infinity we mean the interior of it and we choose, in any chamber $c$ at infinity, an ideal point $\qx_c$, called its \textbf{barycenter}. Moreover, for any two opposite ideal chambers $c$ and $c_{-}$, we impose that the corresponding barycenters $\qx_{c}$ and $\qx_{c_{-}}$ are also opposite. If $G$ is a group of vectorially Weyl-automorphisms of $\QD$, we may suppose that it permutes these barycenters.

\begin{definition}
\label{def::standard_open_neigh_cone_top}
Let $\Delta$ be an affine ordered hovel, $x \in \Delta$ be a point and $c \in  \Ch(\bd \Delta)$ be a chamber at infinity. Let $\xi_c$ be the barycenter of $c$. By the definition of a hovel, we know that there exists an apartment $A \subset \Delta$ such that $x \in A$ and $c \in \Ch(\bd A)$. As $A$ is an affine space, consider the 
ray $[x, \xi_c) \subset A$ issuing from $x$ and corresponding to the barycenter $\xi_c$ of c. 
Let $r\in [x, \xi_c)$ and we define the following subset of $\Ch(\bd \Delta)$
$$
U_{x, r, c}:= \{ c' \in \Ch(\bd \Delta) \; \vert \; [x, r]
\subsetneq [x, \xi_{c'}) \cap [x, \xi_{c})\}.
$$ The subset $U_{x, r, c}$ is called a \textbf{standard open neighborhood} in $Ch(\QD)$ of the (open) chamber $c$ with base point $x$ and gate $r$. 
\end{definition}

\begin{definition}
\label{def::cone_top_on_chambers}
Let $\Delta$ be an affine ordered hovel and let $x$ be a point of $\Delta$. The \textbf{cone topology} $\Top_{x}( \Ch(\bd \Delta))$ on $\Ch(\bd \Delta)$, with base point $x$, is the topology generated by the standard open neighborhoods $U_{x, r, c}$, with $c \in \Ch(\bd )$ and $r\in [x, \xi_c)$.
\end{definition}

\par Actually, $c'\in  U_{x, r, c}$ means that $Q_{x,c}\cap Q_{x,c'}\supset cl([x,r])$ 
 is a ``big'' part of $Q_{x,c}$. It is easy therefore, to see that the topology $\Top_{x}( \Ch(\bd \Delta))$ does not depend of the choice of the barycenters.

\begin{proposition}
\label{prop::independence_base_point}
Let $\Delta$ be an affine ordered hovel and let $x,y \in \Delta$ two different points. Then the cone topologies $\Top_{x}( \Ch(\bd \Delta))$ and $\Top_{y}( \Ch(\bd \Delta))$ are the same.
 \end{proposition}

\begin{proof}
To prove that $\Top_{x}( \Ch(\bd \Delta))$ and $\Top_{y}( \Ch(\bd \Delta))$ are the same, it is enough to show that the  identity map $ \Id : (\Ch(\bd \Delta), \Top_{x}(\Ch(\bd \Delta))) \to (\Ch(\bd \Delta), \Top_{y}(\Ch(\bd \Delta))) $ is continuous with respect to the corresponding topologies. For this it is enough to prove that for every chamber $c \in \Ch(\bd \Delta)$ and every standard open neighborhood $V$ of $c$ in $\Top_{y}(\Ch(\bd \Delta))$, there exists a standard open neighborhood $W$ of $c$ in $\Top_{x}(\Ch(\bd \Delta))$ such that $W \subset V$. 

Let us fix a chamber $c \in \Ch(\bd \Delta)$ and let $V:=U_{y, r, c}$ be a standard open neighborhood of $c$ with respect to the base point $y$, where $r>0$.

\par Let $Q_{z,c}$  be the sector with base point $z \in \Delta$ corresponding to the chamber at infinity $c \in \partial \Delta$. Notice that, for any two points $z_1,z_2 \in \Delta$, the intersection $Q_{z_1,c} \cap Q_{z_2,c}$ is not empty and moreover, it contains a subsector $Q_{z_3, c}$, where $z_3$ is a point in the interior of $Q_{z_1,c} \cap Q_{z_2,c}$. Take $z_1:=x, z_2:=y$ and $z_3:=z$, such that $z$ is in the interior of the intersection $Q_{y,c} \cap Q_{x,c}$.

To $y$ and $z$ as above apply Lemma~\ref{lem::independence_base_point}, that is stated below. We obtain the existence of a standard open neighborhood $U_{z, r', c}$ of $c$ with base point $z$ such that $U_{z, r', c} \subset U_{y, r, c}$. 
 Now for $R:=r' \in [z, \xi_{c} )$ far enough, apply Lemma~\ref{lem::independence_base_point} to the points $z$ and $x$. We obtain the existence of a standard open neighborhood  $U_{x, R', c}$ of $c$ with base point $x$ such that $U_{x, R', c} \subset U_{z, r', c}=U_{z, R, c}$. From here we have that $U_{x, R', c} \subset U_{y, r, c}$ and the conclusion follows.
\end{proof}

\begin{lemma}
\label{lem::independence_base_point}
Let $\Delta$ be an affine ordered hovel and let $c \in \Ch(\bd \Delta)$ be a chamber at infinity. Let $z_1,z_2 \in \Delta$ be two different points such that the sector $Q_{z_1,c}$ contains the point $z_2$ in its interior. Then, for every standard open neighborhood $U_{z_1,r,c}$ of $c$, with base point $z_1$, there exists a standard open neighborhood $U_{z_2,r',c}$ of $c$ with base point $z_2$ such that $U_{z_2,r',c} \subset U_{z_1,r,c}$. And vice versa, for every standard open neighborhood $U_{z_2,R,c}$ of $c$ with base point $z_2$ there exists a standard open neighborhood $U_{z_1,R',c}$ of $c$ with base point $z_1$ such that $U_{z_1,R',c} \subset U_{z_2,R,c}$.
\end{lemma}

\begin{proof}
Let $A$ be an apartment of $\Delta$ such that $Q_{z_1,c} \subset A$. Let $\xi_c$ be the barycenter of the chamber at infinity $c$ and consider the standard open neighborhoods $U_{z_1,r,c}$ and $U_{z_2,R,c}$.  

\medskip
Our first claim is that there exists $r' \in (z_2,\xi_c)$ far enough, such that the enclosure $cl_{A}([z_1,r'])$ contains the segment $[z_1, r)$ and also that there exists $R' \in (z_1, \xi_c)$, far enough, such that $cl_{A}([z_2,R']) \supset [z_2,R]$.  Indeed, this is true because for every point $x$ in the interior of $Q_{z_1,c}$ we have that $cl_{A}([z_1,x]) \supset (Q_{z_1,c} \cap Q_{x,c_{-}})$, where $c_{-} \in \Ch(\bd A)$ denotes the chamber opposite $c$ and $[z_1,x]$ is the geodesic segment with respect to the affine space $A$. By taking $r'$, and respectively, $R'$ sufficiently far enough, the claim follows. 

\medskip
Let us prove the first assertion. Let $r' \in (z_2,\xi_c)$ satisfying the above claim and let $r'' \in (z_1,\xi_c)$ such that $[z_1,r'']= Q_{z_1,c} \cap Q_{r',c_{-}} \cap [z_1,\xi_c)$. This implies that $U_{z_1,r'',c} \subset U_{z_1,r,c}$. Next we want to prove that in fact $U_{z_2,r',c}\subset U_{z_1,r'',c} \subset U_{z_1,r,c}$.

Let $c'\in U_{z_2,r',c}$; so $[z_2,r']\subsetneq [z_2,\xi_{c'})\cap[z_2,\xi_c)$.
We may apply Lemma~\ref{lem::rays_line} (see below) to $\delta_2=[z_2,\xi_{c'})$ and $\delta_1=[r',\xi_{c_{-}})$.
We get an apartment $B$ containing these two rays, hence, their enclosures $Q_{z_2,c'}$ and $Q_{r',c_{-}}$. Thus, $c_{-},c' \in \Ch(\partial B)$, and  $B$ contains $cl_A(z_1,r')\supset Q_{z_1,c}\cap Q_{r',c_{-}}$.
In $B$ all rays of direction $\xi_{c'}$ are parallel to $[z_2,\xi_{c'})$, hence to $[z_2,r']$; or $[z_1,r'']\subset A\cap B$.
So $[z_1,r'']\subset[z_1,\xi_{c'})$ and $c'\in U_{z_1,r'',c} $.
This concludes that $U_{z_2,r',c}\subset U_{z_1,r'',c} \subset U_{z_1,r,c}$.

Let us prove the second assertion. From our first claim there exists $R' \in (z_1, \xi_c)$, far
enough, such that $cl_{A}([z_1,R']) \supset (Q_{z_1,c} \cap Q_{R',c_{-}}) \supset [z_2,R]$. 
If $c'\in U_{z_1,R',c}$, we apply Lemma~\ref{lem::rays_line} to $\delta_2=[z_1,\xi_{c'})$ and $\delta_1=[R',\xi_{c_{-}})$ and one easily proves that $c'\in U_{z_2,R,c} $, hence $U_{z_1,R',c} \subset U_{z_2,R,c}$.
\end{proof}

\begin{lemma} 
\label{lem::rays_line}
Let $\delta_1,\delta_2$ be two preordered rays in apartments $A_1,A_2$ of an affine ordered hovel $\Delta$, with origins $x_1,x_2$. Suppose $x_1\not=x_2$ and $x_1,x_2\in \delta_1\cap\delta_2$ (hence $[x_1,x_2]\subset \delta_1\cap\delta_2$). Then $ \delta_1\cup\delta_2$ is a line in an apartment $A$ of $\Delta$.
\end{lemma}

\begin{proof} This is exactly what is proved in part 2) of the proof of Rousseau~\cite[Prop. 5.4]{Rou11}.
\end{proof}

\begin{lemma} \label{3.6} Let $Q_1,Q_2$ be two sectors in an affine ordered hovel $\QD$, sharing the same base point $x$, with $Q_1$ (resp.,  $Q_2$) of positive (resp., negative) direction.
We suppose $Q_1,Q_2$ are opposite, \ie there exist $y_1\in Q_1,y_2\in Q_2$ (hence $y_2\stackrel{o}{<} x\stackrel{o}{<} y_1$) such that $x\in[y_1,y_2]$.
Then there is an apartment $A$ containing $Q_1$ and $Q_2$.
\end{lemma}

\begin{rema*} 
\label{rem::cond_CO}
This is condition (CO) of Parreau~\cite[1.12]{Parr00}
\end{rema*}

\begin{proof} Let $\qd'_i$ be the generic ray of origin $x$ in $Q_i$ containing $y_i$.
By (MA3) there is an apartment $A_i$ containing $germ(\qd'_i)$ and $germ_x([x,y_{3-i}])$, hence also $\qd'_i$ (by convexity) and some point $x_i\in (x,y_{3-i}]$.
Now it is clear that $\qd_i=\qd'_i\cup[x,x_i]$ is a (generic) ray in $A_i$. 
Applying Lemma \ref{lem::rays_line}, there is an apartment $A$ containing $\qd_1$ and $\qd_2$.
Now, by (MA2) and (MA3), $A$ contains $cl_A(\qd'_i)\supset Q_i$.
\end{proof}

\section{A criterion for strong transitivity}\label{s4}

In this section we verify an analogue criterion for strong transitivity as in Caprace--Ciobotaru~\cite[Section 3]{CaCi}. We first start with a definition.

\begin{definition}
\label{def::good_part}
Let $\Delta$ be an affine ordered hovel and let $G \leq \Aut(\Delta)$ be a group of automorphisms. Let $A$ be an apartment of $\Delta$. We say that a subset (or a filter) $\Omega$ of the apartment $A$ is \textbf{good} with respect to the group $G$ if, for any two apartments $A,A'$ containing $\QO$, there is $g\in G_\QO^{0}:=\{ g\in G \; \vert \;  g \text{ fixes pointwise } \QO \}$ such that $A'=g(A)$ and $g : A\to A',x\mapsto g(x)$ is a Weyl-isomorphism.
\end{definition}

\par For $\QO$ a local chamber or a sector germ, $\QO$ is good if, and only if, $G_\QO^{0}$ is transitive on the set of all apartments containing $\Omega$: this is a consequence of (MA2) and (MA4) as two isomorphisms $A\to A'$ fixing $\QO$ are necessarily equal; this implies in particular that the corresponding isomorphism is Weyl.  In the same way, we also obtain that $G_\QO^{0}$ fixes pointwise $\overline\QO$ and $cl_A(\QO)$.

\begin{definition}
\label{4.2}
(See Gaussent--Rousseau~\cite[1.5]{GR14} or Rousseau~\cite[4.10]{Rou13}). Let $\Delta$ be an affine ordered hovel and let $G$ be a subgroup of $\Aut(\Delta)$. We say that $G$ acts {\bf strongly transitively} on $\Delta$ if any isomorphism involved in axioms (MA2), (MA4) may be chosen to be the restriction of an element of $G$.
\end{definition}

\par This means that each of the sets $cl_A(F)$ appearing in (MA2) and $cl_A(\g R\cup F)$ appearing in (MA4) is good with respect to the group $G$.

\par The Kac--Moody groups as in the example \ref{1.3}.5) act strongly transitively by vectorially Weyl-automorphisms on the corresponding hovels.

\begin{proposition}
\label{lem:ST:criterion}
Let  $\Delta$ be an 
affine ordered hovel and let $G \leq \Aut(\Delta)$ be a group of vectorially Weyl-automorphisms. Then the following conditions are equivalent. 

\begin{enumerate}[(i)]
\item $G$ is strongly transitive on $\Delta$;
\item
every local chamber of $\Delta$ is good;
\item
every sector-germ of $\Delta$ is good.
\end{enumerate}
\end{proposition}

\begin{proof}
(i) $\Rightarrow$ (iii): This follows from the definition of the strongly transitive action of $G$ on the hovel $\Delta$.

\medskip
(ii) $\Rightarrow$ (i):
First we claim that every local face $F$ of the hovel $\Delta$ is good. Indeed, if there exist two apartments $A$ and $A'$ such that $F \subset A \cap A'$, we consider a local chamber $C$ of $A$ and a local chamber $C'$ of $A'$ such that both cover $F$. 
By Rousseau~\cite[Prop. 5.1]{Rou11} there is an apartment $A''$ of $\Delta$ containing $\overline C \cup \overline C'$. 
 Applying our hypothesis, the claim follows.

\medskip
We need to verify that all isomorphisms involved in the definition of the hovel are induced by elements of $G$.  Therefore, let $\widetilde\QO=cl_A(F)$ or $\widetilde\QO=cl_A(\g R\cup F)$ be as in the axiom (MA2) or (MA4). We consider a closed subset $\QO\subset supp(\widetilde\QO)$ which is an element of the filter $\widetilde\QO$ and
is contained in the intersection $A \cap A'$ of two apartments $A, A'$ of $\Delta$.
 But in $\Omega$ one can find a (maximal) local face $F_1$ such that $F_1 \subset \Omega \subset supp(F_1)$, where $supp(F_1)$ is the unique affine space of minimal dimension that contains $F_1$. For this face $F_1$ we apply our above claim and we obtain an element $g$ in the pointwise stabilizer $G_{F_1}^{0} < G$ such that $g(A)=A'$, $g : A\to A',x\mapsto g(x)$ is a Weyl-isomorphism and $g$ fixes pointwise the face $F_1$. As $F_1 \subset \Omega$ and $g(A)=A'$ we also have that $g$ fixes pointwise the subset $\Omega$. The conclusion follows. 

\medskip
(iii) $\Rightarrow$ (ii):
Let $C$ be a local chamber contained in the intersection of two apartments $A$ and $A'$ of $\Delta$.
 Let $x$ be the vertex  of $C$ and consider in $A$ the sector $Q_x$ with base point $x$ that contains the chamber $C$. 
 Then $A\cap A'$ contains a neighborhood $C'$ of $x$ in $Q_x$, so $C\subset C'$.
 Let $y$ be a point in the interior of $C'$. 
 Take the sector in $A$ of the form $Q_y:= Q_x+ (y-x) \subset Q_x$. 
 In the apartment $A'$ consider the sector $Q_y'$ with base point $y$  that contains the vertex $x$, hence which is opposite  $Q_y$.
  Notice that $Q'_y \cap Q_x$ is a small neighborhood of $x$ (respectively, of $y$). 
  By applying 
 Lemma~\ref{3.6}, one concludes that there is an apartment $A''$ containing the sectors $Q_x$ and $Q'_y$. 
 So $A''$ contains $Q_x\subset cl_A(\{x\}\cup Q_y)$ and $Q_x\cap Q'_y$ is an element of the filter $C$.
  By the hypothesis, applied successively to $(S_x,A,A'')$ and $(S'_y,A',A'')$, we obtain the conclusion.

\end{proof}

\section{Existence and dynamics of strongly regular elements}
\label{sec::str_reg_elements}

\begin{definition}(See~\cite[Section~2.1]{CaCi})
\label{def::str_reg_lines}
Let $(\Delta, \sha)$ be an affine ordered hovel and let $A$ be an apartment in $\sha$. With respect to the affine structure of $A$, a line $\ell \subset A$ is called \textbf{strongly regular}, if its points at infinity lie in the interior of two opposite chambers of the twin building at infinity of $\Delta$. This also means that both associated ray germs are generic. In particular, the apartment $A$ is the unique apartment of $\Delta$ containing the strongly regular line $\ell$.
 
 Moreover, a hyperbolic element $\gamma$ of $\Aut(\Delta)$ is called \textbf{strongly regular} if it admits a strongly regular translation axis (i.e., there exists an apartment $A $ of $\Delta$ and a strongly regular geodesic line in $A$ which is a translation axis of $\gamma$). 

\end{definition}

\begin{theorem}
\label{thm:ExistenceStronglyReg}
Let $(\Delta, \sha)$ be a thick, affine, semi-discrete, ordered hovel such that $\A^v=(V,W^v)$ has no irreducible factor of affine type \red{and such that $\Delta$ is not a tree}. Let $G$ be a vectorially Weyl subgroup of $\Aut(\Delta)$ that acts strongly transitively on the twin building at infinity $\partial\QD$ of $\Delta$ and \red{satisfies condition (LST)}.
Then $G$ contains a strongly regular hyperbolic element. 
\end{theorem}

\begin{remark}
\label{rem::AGT} As stated, the theorem fails when $\QD$ is of affine type: we may consider a split loop group $G$ over a non-Archimedean local field, acting on its hovel. It acts strongly transitively on $\partial\QD$, but the smallest positive imaginary root $\qd\in\QF^+_{im}$ is trivial on the maximal torus $T$. So all translations induced by $\Stab_G(A)$ on the corresponding apartment $A$ have their corresponding translation vectors in the boundary $\ker(\qd)$ of the Tits cone: there is no strongly regular element.
 
 \par It will be clear from the proof of the theorem that the conclusion of the theorem is true even when $\A^v$ may have an irreducible factor of affine type, if we add the following hypothesis for the action of $G$:
 
 \medskip\par (AGT) For some (and hence for any) apartment $A$ of the hovel $\Delta$, its stabilizer $\Stab_G(A)$ in $G$ contains an element which induces an \textbf{affinely generic translation}, i.e., a translation whose corresponding translation vector $\vect v$ satisfies $\qd(\vect v)\neq 0$ for any imaginary root of an affine component of $\QF\cup\QF_{im}$.
\end{remark}

\begin{proof}[Proof of Theorem~\ref{thm:ExistenceStronglyReg}]
As $G$ is strongly transitive on the twin building at infinity of $\Delta$, we have that $G$ acts transitively on the set of all apartments of $\Delta$.
If 

We start with a preliminary observation. Let $A$ be an apartment of $\Delta$ and let $H, H'$ be two complementary half-apartments of $A$. We claim that there is some $g \in \Stab_G(A)$ which swaps $H$ and $H'$; in particular $g$ stabilizes the common boundary wall $\partial H = \partial H'$. 

In order to prove the claim, choose a pair of opposite panels $\sigma, \sigma'$  at infinity of the wall $\partial H$. Notice that it makes sense to consider panels at infinity since $\Delta$ is not a tree, and thus the twin building $\partial \Delta$ has positive rank. By Lemma~\ref{lem:OppPanels}, the stabilizer $G_{\sigma, \sigma'}$ acts $2$-transitively on $\Ch(\sigma)$. We now invoke Lemma~\ref{lem:TreeWalls} which provides a canonical isomorphism $P(\sigma, \sigma') \cong (T(\sigma, \sigma'), \sha(\sigma, \sigma')) \times \RR^{n-1}$, where $n = \dim(\Delta)$ and $(T(\sigma, \sigma'), \sha(\sigma, \sigma'))$ is a thick tree. 
The set $\Ch(\sigma)$ being in one-one correspondence with $ \partial T(\sigma, \sigma')$, we infer that $G_{\sigma, \sigma'}$ is $2$-transitive on $\partial T(\sigma, \sigma')$. \red{Recall that $G_{\sigma, \sigma'}$ is strongly transitively on $(T(\sigma, \sigma'), \sha(\sigma, \sigma'))$ by \blue{the hypothesis (LST)}. We emphasize that one cannot apply Caprace--Ciobotaru~\cite[Corollary~3.6]{CaCi} as in our case the apartment system $\sha(\sigma, \sigma')$ is not necessarily complete \blue{if $\sha$ is not locally complete}.}
Therefore, if $v$ (resp., $D$) denotes the vertex (resp., geodesic line) of $T$ corresponding to the wall $\partial H$ (resp., the apartment $A$), we can find some $g \in \Stab_{G_{\sigma, \sigma'}}(A)$ stabilizing $D$ and acting on it as the symmetry through $v$. It follows that $g$ stabilizes $A$, the wall $\partial H  = \partial H'$ and swaps the two half-apartments $H$ and $H'$. This proves our  claim. (We warn the reader that $g$ might however act non-trivially on the wall $\partial H$.)  

\medskip
Using the above proven fact, we claim that for every root $\alpha \in \QF$, we can construct a translation in $\Stab_G(A)$ of translation vector that is almost collinear to $\alpha^\vee$. 

Indeed, let $h$ be a wall of $A$ and denote by $H,H'$ the corresponding two complementary half-apartments of $A$ such that $\partial H = \partial H'=h$. Let $g \in \Stab_{G}(A)$  that swaps $H$ and $H'$ and stabilizes the wall $h$ together with a panel $\sigma$ at infinity of $h$. 
We have that  $r_h\circ g$ is a vectorially Weyl-automorphism of $A$ and stabilizes the two chambers of $\partial A$ containing $\sigma$, where $r_h \in  \Stab_{\Aut(\Delta)}(A) \leq \Aut(\Delta)$ denotes the reflection with respect to the wall $h$. We conclude that $r_h\circ g$ is a translation in $\Stab_{\Aut(\Delta)}(A)$. 
As $g$ and $r_h$ stabilize $h$, the element $r_h\circ g$ is a translation parallel to $h$. We call such an element $g  \in \Stab_{G}(A)$ a \textbf{reflection-translation} of wall $h$.

\medskip
We remark the following. Fix two different walls $h_1$ and $h_2$ of the apartment $A$ of direction $\ker\alpha$ and let $g_1, g_2 \in \Stab_{G}(A)$ be two reflection-translations corresponding respectively, to the walls $h_1$ and  $h_2$. We have that $g_1 \circ g_2 \in \Stab_{G}(A)$ is a translation element whose translation vector is not in $\ker\alpha$. 
In particular, we notice that the projection 
of the translation vector of $g_1 \circ g_2$ in the direction $\ker\alpha$ equals the sum of the translation vectors  of $g_1$ and $g_2$. 
Moreover, the projection of the translation vector of $g_1 \circ g_2$ in the direction of $\alpha^\vee$ depends on the euclidean distance between the walls $h_1$ and $h_2$.

 Let us denote $\gamma: =g_1 \circ g_2 $. For every reflection-translation $g \in \Stab_{G}(A)$ of wall $h_g$ of direction $\ker\alpha$, we have that $\gamma^n g \gamma^{-n} \in \Stab_{G}(A)$ is a reflection-translation of wall $\gamma^n (h_{g})$ and whose translation vector equals the translation vector of $g$. 
 We conclude that every wall of direction $\ker\alpha$ admits a reflection-translation in $\Stab_{G}(A)$ whose translation vector is bounded  independently of the wall. To conclude our last claim, one can choose two walls $h_1$ and $h_2$ of direction $\ker\alpha$ that are very far away and two reflection-translations $g_1, g_2 \in \Stab_{G}(A)$ corresponding respectively to $h_1$ and $h_2$. 
 By the remark above, we have that $g_1 \circ g_2 \in \Stab_{G}(A)$ is a translation and the projection of its translation vector in the direction  $\alpha^\vee$ can be made very big, depending on the distance between $h_1$ and $h_2$. 
 In particular, the translation vector of $g_1 \circ g_2$ can be made almost collinear to $\alpha^\vee$ as we want. The claim follows.
 
Finally, to construct a strongly regular hyperbolic element, we choose a base of roots 
 $(\qa_i)_{i\in I}$.
By the hypothesis on the type of $(V,W^v)$, we get some $\Z-$linear combination $\sum_{i\in I}\,n_i\qa_i^\vee$ which is in $C^v_f$, see Kac~\cite[Th.~4.3]{K90}.
But from above, we get a fixed neighborhood $U$ of $0$ in $V$ and elements $g_{i,n}\in \Stab_G(A)$ inducing translations with vector in $nr_i\qa_i^\vee+U$, for every $i \in I$ and some $r_i\in(0,+\infty)$.
Therefore, for some $N_i$ big (with all $N_ir_i/n_i$ almost equal), the product $\prod_{i\in I}\,g_{i,N_i}$ (in any order) will be the desired hyperbolic element.
\end{proof}

\begin{remark}
\label{rem::semi-discrete}
The hypothesis that $\Delta$ is a semi-discrete hovel is necessary in order to apply, along the proof of Theorem~\ref{thm:ExistenceStronglyReg}, Corollary~3.6 from Caprace--Ciobotaru~\cite{CaCi}. This corollary regards only thick trees and not real trees.
\end{remark}


\begin{proposition}(See~\cite[Proposition~2.12]{CaCi})
\label{prop::dynamics_str_reg}
Let $\Delta$ be an affine ordered hovel and let $\gamma \in \Aut(\Delta)$ be a vectorially Weyl strongly regular hyperbolic element of translation apartment $A$. Let $c_{-} \in \Ch(\bd A)$ be the unique chamber at infinity that contains the repelling point of $\gamma$ in its interior. 

Then for every $c \in \Ch(\bd \Delta)$ the limit $\lim\limits_{n \to \infty} \gamma^{n}(c)$ exists with respect to the cone topology on $\Ch(\bd \Delta)$ and coincides with the retraction of $c$ onto $A$ based at the chamber $c_{-}$. In particular, the fixed-point-set of $\gamma$ in $\Ch(\bd \Delta)$ is the set $\Ch(\bd A)$.
\end{proposition}
\begin{proof}
The proof goes in the same way as in Caprace--Ciobotaru~\cite[Proposition~2.10]{CaCi}.
\end{proof}


\section{Main theorem}
\label{sec::main_theorem}

\begin{theorem}
\label{thm::main_theorem}
Let $(\Delta, \sha)$ be a semi-discrete thick, affine ordered hovel \red{that is not a tree.} 
Let $G$ be a vectorially Weyl subgroup of $\Aut(\Delta)$ \red{satisfying condition (LST)} and with the property that for every $c \in \Ch(\partial \Delta)$ the $G^{0}_c$--orbits on $\Opp(c)$ are closed with respect to the cone topology on $\Ch(\partial \Delta)$. 
If $\QD$ has some factors of affine type, we ask moreover that $G$ satisfies  the condition (AGT) of~\ref{rem::AGT}. 
Then the following are equivalent:

\begin{enumerate}[(i)]
\item
\label{thm:main-thm-i}
$G$ acts strongly transitively on $\Delta$;

\item
\label{thm:main-thm-ii}
$G$ acts strongly transitively on the twin building at infinity $\partial \Delta$.
\end{enumerate}
\end{theorem}

\begin{remark} 
\label{rem::semi-discretness}
The semi-discreteness is used only for (ii) $\Rightarrow$ (i).
\end{remark} 

\begin{proof}[Proof of Theorem~\ref{thm::main_theorem}] (i) $\Rightarrow$ (ii) 
\medskip
By the axiom (MA3), it is clear that $G$ is transitive on the set of all apartments of $\Delta$. It remains to prove that the stabilizer $\Stab_{G}(A)$ of an apartment $A$ of $\Delta$ is transitive on the set of $\Ch(\partial A)$. To obtain this it is enough to prove that the affine Weyl group $W^{a}$ is contained in $\Stab_{G}(A)$. 

First, consider three half-apartments $H_1,H_2, H_3 \subset \Delta$ that have a wall $h$ in common and such that $H_i \cap H_j=h$, for every $i\neq j$. We claim there exists an element $g \in G$ such that $g$ fixes $H_1$ and $g(H_2)=H_3$. In particular, we obtain that $g(A_2)=A_3$, where $A_2=H_1\cup H_2$ and $A_3=H_1 \cup H_3$. Indeed, let $Q_1$ be a sector in $H_1$ such that it admits a sector-panel-germ at infinity of $h$, which is denoted by $\mathfrak{f}_1$. Consider at infinity of $h$ a sector-panel-germ $\mathfrak{f}_2$ that is opposite $\mathfrak{f}_1$. Let $Q_2$ (resp., $Q_3$) be a sector in $H_2$ (resp., $H_3$) that contains $\mathfrak{f}_2$ at infinity. One notices that $Q_1$ and $Q_2$ (resp., $Q_3$) are of opposite direction in $A_2$ (resp., $A_3$). We apply axiom (MA4) for the $germ(Q_1)$ and $\mathfrak{f}_2$ and to the apartments $A_2$ and $A_3$. Then, there exists a Weyl-isomorphism $g \in G$ such that $g(A_2)=A_3$ and $g$ is fixing pointwise $cl_{A_2}(germ(Q_1),\mathfrak{f}_2)$. Notice that $ H_1 = cl_{A_2}(germ(Q_1),\mathfrak{f}_2)$; therefore, $g(H_1)=H_1$ pointwise. The claim is proved.

Let now $h$ be a wall of the apartment $A$ and denote by $H_1$, $H_2$ the half-apartments of $A$ such that $\partial H_1=\partial H_2=h$. We claim that there exists $g \in  W^{a} \leq \Stab_{G}(A)$ such that $g$ is a reflection with respect to the wall $h$. Indeed, as the hovel is thick, by Rousseau~\cite[Prop.~2.9]{Rou11} there exists a third half-apartment $H_3 \subset \Delta$ such that $H_1 \cap H_3= H_2 \cap H_3$. From the above claim, applied twice, we obtain an element $g \in \Stab_{G}(A)$ with the desired properties. In particular, as $g$ fixes pointwise the wall $h$, we have that $g \in W^{a}$ and the conclusion follows. 

\medskip \noindent 
(ii) $\Rightarrow$ (i) 

Let $c \in \Ch(\partial \Delta)$, $c' \in \Opp(c)$ and denote by $A$ the unique apartment in $\Delta$ whose boundary contains $c$ and $c'$.

As $G$ is strongly transitive on $\partial \Delta$, by Theorem~\ref{thm:ExistenceStronglyReg} we have that $G$ contains a strongly regular hyperbolic element. As $G$ acts transitively on the set of all apartments of $\Delta$, we conclude that every apartment in $\Delta$ admits a strongly regular hyperbolic element in $G$. Moreover, there exists a strongly regular hyperbolic element $\gamma$ of $ G_{c,c'} \leq G$ such that $c$, respectively $c'$, is the unique chamber at infinity that contains in its interior the repelling, respectively the attracting, point of $\gamma$.

Applying Proposition~\ref{prop::dynamics_str_reg} to the strongly regular element $\gamma$ with its unique translation apartment $A$ we obtain that $c'$ is an accumulation point for every $G^0_{c} G_{c, c'}$--orbit in $\Opp(c)$, with respect to the cone topology on $\Ch(\partial \Delta)$. By our hypothesis, every $G^{0}_c$--orbit in $\Opp(c)$ is closed in the cone topology on $\Ch(\partial \Delta)$, and so is every $G^0_{c} G_{c, c'}$--orbit in $\Opp(c)$, by Lemma~\ref{lem:Levi}. Therefore, there is only one $G^0_{c} G_{c, c'}$--orbit in $\Opp(c)$. We conclude that $G_{c}=G^0_{c} G_{c, c'}$ and that $G^0_{c} $ is transitive on $\Opp(c)$. The desired conclusion follows from the criterion of strong transitivity given by Proposition~\ref{lem:ST:criterion}.

\end{proof}

\end{document}